\documentclass{article}

\usepackage{PRIMEarxiv}
\usepackage{amsmath}
\usepackage{graphicx,epstopdf}
\usepackage[utf8]{inputenc} 
\usepackage[T1]{fontenc}    
\usepackage{hyperref}       
\usepackage{url}            
\usepackage{booktabs}       
\usepackage{amsfonts}       
\usepackage{nicefrac}       
\usepackage{microtype}      
\usepackage{lipsum}
\usepackage{amsthm}
\usepackage{amssymb}
\usepackage{fancyhdr}       
\usepackage{graphicx} 
\newtheorem{theorem}{Theorem}
\numberwithin{equation}{section}
\usepackage{subfig}
\graphicspath{{media/}}     

\title{Dirichlet-Neumann and Neumann-Neumann Waveform Relaxation Methods for PDEs with Time Delay
\thanks{Under review} 
}

\author{
Bankim Chandra Mandal, Deeksha Tomer, \\
  Indian Institute of Technology Bhubaneswar\\
  \texttt{\{a21ma09002, bmandal\}@iitbbs.ac.in} \\
}

\begin{document}
\maketitle

\begin{abstract}
We introduce and compare two domain decomposition based numerical methods, namely the Dirichlet-Neumann and Neumann-Neumann Waveform Relaxation methods (DNWR and NNWR respectively), tailored for solving partial differential equations (PDEs) incorporating time delay. Time delay phenomena frequently arise in various real-world systems, making their accurate modeling and simulation crucial for understanding and prediction. We consider a series of model problems, ranging from Parabolic, Hyperbolic to Neutral PDEs with time delay and apply the iterative techniques DNWR and NNWR for solving in parallel. We present the theoretical foundations, numerical implementation, and comparative performance analysis of these two methods. Through numerical experiments and simulations, we explore their convergence properties, computational efficiency, and applicability to various types of PDEs with time delay.
\end{abstract}


\section{Introduction}

In many real-world scenarios, time lags often exist between actions taken and their resulting responses. 
For example, in medical science \cite{schiesser2019time}, the delay is the time it takes for symptoms to show up after someone gets infected, called the incubation period. In population studies \cite{kuang}, delays occur because individuals only start reproducing after reaching a certain age and in control theory, delays happen because signals and technology processes take time to happen.
Delay Partial Differential Equations (Delay PDEs) serve as valuable modelling tools that offer a more realistic representation of processes characterized by time lags\cite{schiesser2019time,kuang,Wu,hetzer1995functional,hetzer1996global,covid}. 

 Thus, Delay PDEs are widely used in real-world applications across various fields, including biology, control theory, medicine, and climate modelling
 (see \cite{Wu}). In a sequence of papers, Hetzer and others \cite{hetzer1995functional, hetzer1996global} have shown that reaction-diffusion equations with delays emerge organically in the investigation of climate models. Cheng-Cheng Zhu and Jiang Zhu \cite{covid} have developed a delay reaction-diffusion model that closely examines the real-world spread of the COVID-19 epidemic. This model incorporates factors such as relapse, time delays, home quarantine, and a spatio-temporal heterogeneous environment, all of which play significant roles in influencing the spread of COVID-19. For modelling population dynamics, the delay diffusion model is used (see \cite{population}). Wave equations incorporating delay terms serve as fundamental tools for analyzing oscillatory phenomena \cite{oscillation}, including aftereffects,
time lags or hereditary characteristics 
\cite{Wang,Liu,Wu}, as the deformation of visco-elastic materials \cite{viscoelasticity} or the retarded control of the
dynamics of flexible structures \cite{flexible,beam}. Time delays naturally occur in different communication systems like satellite links, wireless networks, or internet protocols. Utilizing PDEs that incorporate time delay enables the modeling and analysis of these systems, aiming to enhance data transmission efficiency, optimize network performance, and minimize latency \cite{forouzan2007data}.

To solve Delay PDE, we generally apply the separation of variables technique to derive Fourier series solutions in spatial variables for linear initial-boundary value problems involving parabolic and hyperbolic PDEs with constant or proportional delay, as well as various boundary conditions \cite{polyanin2023delay}. 
Numerical methods for solving Delay PDEs include Method of lines which involves converting a Delay PDE into a set of Delay ODEs. Finite-difference methods, employing implicit schemes and weighting schemes, are commonly used to numerically solve Delay PDEs \cite{polyanin2023delay}. However, there remains a lack of efficient numerical methods for solving these equations, especially in the context of parallel computational techniques. Parallel approaches play a vital role in solving extremely large-scale numerical problems in a time-bound manner. Domain Decomposition is one of the important techniques, used nowadays for parallel computation of solutions for differential equations in general \cite{Ong2016PipelineIO,kwok2019schwarz}. These methods involve dividing the computational domain into multiple smaller subdomains. The approach then involves separating the problem by providing an initial estimate along the interfaces of these subdomains and solving each subdomain problem simultaneously. Throughout each iteration, the method monitors for inconsistencies, such as non-smooth transitions, along the interfaces between subdomains. This iterative process continues until a uniform and smooth solution across the entire domain is attained. 
 The domain decomposition approach was first introduced in 1870 by Schwarz \cite{schwarz}, who demonstrated the existence of solutions to the Laplace equation in arbitrary domains. This was achieved by dividing the domain into overlapping subdomains and sequentially computing solutions within each subdomain, ensuring convergence using the maximum principle. P.L Lions \cite{lions1,lions2,lions3} later extended this method by proposing a parallel algorithm that modified the transmission conditions from the original Schwarz approach.

DD methods for steady problems include overlapping Classical Schwarz and Optimized Schwarz methods, while non-overlapping DD methods for steady problems include Dirichlet-Neumann and Neumann-Neumann methods. Dirichlet-Neumann method was first given by  Bjørstad and Widlund \cite{Bjorstad} while the Neumann-Neumann method was originally considered by  Bourgat et al. \cite{Bourgat}. DD methods for evolution problem include Classical Schwarz Waveform Relaxation (SWR), Optimised Schwarz Waveform Relaxation, Dirichlet-Neumann Waveform Relaxation (DNWR) \cite{gander,ganderparabolic} and Neumann-Neumann Waveform Relaxation (NNWR) \cite{gander,ganderparabolic} methods. DNWR and NNWR are substructuring analogues of Dirichlet-Neumann and Neumann-Neumann methods. DNWR relies on dividing the spatial domain into non-overlapping subdomains. Each iteration then consists of solving subdomain problems in both space and time, along with enforcing interface conditions. Finally, the iteration concludes with an update step. NNWR extends the Neumann-Neumann approach for elliptic problems to the time-dependent problems. In the first phase of the Neumann-Neumann method, subdomain problems are solved with Dirichlet interface conditions. This is followed by a correction step, which incorporates Neumann boundary conditions along the interfaces. Both DNWR and NNWR converge faster in terms of iteration efficiency than Classical Schwarz and Optimised Schwarz for Heat \cite{ganderparabolic} and Wave \cite{gander} equations. Because of their non-overlapping divisions, they allow to use different discretization schemes in different subdomains. To improve computational efficiency and handle large-scale problems by breaking them down into smaller, more manageable subproblems and iteratively refining the solution in parallel, pipeline and adaptive pipeline versions of DNWR and NNWR algorithms are developed and studied in \cite{Ong2016PipelineIO} and \cite{kwok2019schwarz}.

We implement these substructuring Waveform Relaxation methods, namely DNWR and NNWR, to solve Partial Differential Equations with time delay. This methodology was specifically applied to Reaction-Diffusion Equations involving time delay, Neutral PDEs, and Hyperbolic delay PDEs. We will see their convergence behavior and performance numerically and see some theoretical aspects.

Although significant attention has been devoted to the theoretical investigation of Delay PDEs over the years, there remains a gap in implementing numerical methodologies for solving such equations in parallel. Zubik-Kowal and Vandewalle \cite{kowal} analysed the convergence of waveform relaxation methods \cite{Lelarasmee,Svandewalle}, including Gauss-Seidel and Jacobi types, for addressing discretized Delay PDE problems. In (2005), Gander and Vandewalle investigated the overlapping optimized Schwarz domain decomposition method for Parabolic  PDEs incorporating time delay \cite{vandewalle}. In 2012, Shu-Lin conducted a convergence analysis of the overlapping Schwarz waveform relaxation (OSWR) method for reaction-diffusion equations with time delay \cite{shulin, shulinquasi}.

The analysis of NNWR and DNWR for delay problems has not yet been done. In the same direction we make the following novel contribution:
\begin{itemize}
    \item We study the behavior of DNWR and NNWR for Delay PDEs. These methods are effective parallel numerical methods for solving PDEs and haven't been used for Delay PDEs yet.
    \item For Parabolic Delay PDE, the proposed method achieves better convergence as compared to the existing numerical methods.
    \item We also perform experiments for DNWR and NNWR in the case of the wave equation with delay. For a specific relaxation parameter, we demonstrate the convergence of both DNWR and NNWR methods in two iterations for $\theta =1/2$ and $\theta =1/4$, respectively. We also implement these methods for Neutral PDE and observe similar results.
    \item We theoretically prove convergence for an infinite domain for Parabolic PDE and for a symmetric domain in the case of Hyperbolic PDE.
    \item We observe the behaviour of these methods for multiple subdomains (when the domain is divided into more than 2 subdomains).
\end{itemize}
 
 The structure of this paper is as follows: In Section 2, we utilize model equation 1, which represents Parabolic PDE with constant time delay, to implement the DNWR and NNWR algorithms. Section 3 focuses on convergence results for DNWR and NNWR for  Hyperbolic delay PDEs, where we derive Theorem 3 for DNWR and Theorem 4 for NNWR. In Section 4, we conduct an analysis of DNWR and NNWR for Neutral PDEs, resulting in Theorems 5 and 6. Section 5 presents numerical experiments to validate the effectiveness of the theorems obtained in the previous sections and the comparison of DNWR and NNWR for the Parabolic PDE with time delay with the existing Classical Schwarz and Optimised Schwarz methods. This section also includes experiments for multiple subdomain cases. %
 In Section 6, we conclude this paper with some final remarks.
 
 \section{Parabolic PDE with Time Delay}\label{sec1}
As a model problem, we first consider the Reaction-Diffusion Equation with a constant time delay (this equation appears in \cite{shulin}).\\
\begin{equation}\label{eq_1}
\left\{\begin{array}{rl}
\partial u_t-\upsilon ^2\partial ^2u_{xx}+a_1u(x, t)+a_2u(x, t-\tau ) &=f(x, t),\ \ \ (x, t)\in \mathbb{R}\times (0, T), \\ 
 u(x, t) &=u_0(x, t), \ \ \  (x, t) \in \mathbb{R}\times [-\tau , 0], \\ 
 u(\pm\infty , t ) &=0, \ \ \ t\in (0, T) , 
\end{array}\right.
\end{equation}
\\
where $\tau>0$ is the time lag, $ \upsilon>0$ is the diffusion coefficient and $a_1, a_2$ are real constants with $a_2\neq0$.
\subsection{Convergence of DNWR}
 We study the convergence behaviour of DNWR method for the Problem (\ref{eq_1}). Since the problem
 is linear, we consider the corresponding error equations where $f(x,t)=0$ and $u_0(x,t)=0$. The spatial domain is partitioned into two non-overlapping subdomains, denoted as $\Omega_1 =(-\infty, L)$ and $\Omega_2 =(L,\infty)$. Taking an initial guess as $h^0(t)$ on the interface boundary, we first solve a Dirichlet problem in subdomain $\Omega_1$ followed by a Neumann problem in subdomain $\Omega_2$. These two steps constitute a single iteration, and after each iteration k, we update the interface value {$h^k(t)$} for $k=1,2,\ldots$. The DNWR reads as follows:
 \\
Dirichlet Part:
\begin{equation}\label{eq_2}
\left\{\begin{array}{rl}
\partial_t e_1^k-\upsilon ^2\partial_{xx} e_1^k+a_1e_1^k(x, t)+a_2e_1^k(x, t-\tau ) &=0, \ \ \ (x, t)\in \Omega_1\times (0, T),  \\ 
  e_1^k(x, t)&=0, \  \ \ (x, t)\in \Omega_1\times [-\tau, 0], \\ 
e_1^k(-\infty, t)&=0, \ \ \ t\in(0, T),\\ 
e_1^k(L, t)&=h^{k-1}(t), \ \ \  t\in(0, T).
\end{array}\right.
\end{equation}\\
Neumann Part:
\begin{equation}\label{eq_3}
\left\{\begin{array}{rl}
\partial_t e_2^k-\upsilon ^2\partial_{xx} e_2^k+a_1e_2^k(x, t)+a_2e_2^k(x, t-\tau )&=0, \ (x, t)\in \Omega_2\times (0, T),  \\ 
  e_2^k(x, t)&=0, \  (x, t)\in \Omega_2\times [-\tau, 0], \\ 
  \partial_x e_2^k(L, t)&=\partial_x e_1^k(L, t), \  t\in(0, T),\\
e_2^k(\infty, t)&=0,  \  t\in(0, T).

\end{array}\right.
\end{equation}\\
The update condition along the interface is given by 
\begin{equation}\label{eq_33}
h^{k}(t)=\theta e_2^{k}\mid _{\Gamma\times(0, T)} +(1-\theta) h^{k-1}(t).
\end{equation}\\
The parameter $\theta$ is the relaxation parameter lying in $(0,1)$. The main goal is to analyse how the error $h^k(t)$ introduced on the interface boundary decreases and goes to zero as $k\rightarrow \infty$ so that we get a smooth solution on the entire domain.\\


\begin{theorem}

 For the Parabolic PDE (\ref{eq_1}) with time delay, the DNWR algorithm (\ref{eq_2})-(\ref{eq_33}) converges linearly for $0<\theta<1$, $\theta\neq 1/2$. For $\theta=1/2$ it converges within two iterations, irrespective of the time window.
 \label{thm:thmdnwr}
\end{theorem}
\begin{proof}
 We prove the result using the Laplace transform technique. Upon applying the Laplace transform to the DNWR algorithm  (\ref{eq_2})-(\ref{eq_33}), for $k=1,2,3\ldots$, the error equation is transformed into the following form:

\[
\left\{\begin{array}{rl}
(s+a_1+a_2e^{-\tau s})\hat e_1^k-\upsilon ^2\partial _{xx}\hat e_1^k&=0,\quad\text{in}\quad\Omega_1
\\ 
\hat e_1^k(-\infty,s)&=0 ,\\ 
\hat e_1^k(L,s)&=\hat h^{k-1}(s),\\
\end{array}\right.
\]
\[
\left\{\begin{array}{rl}
(s+a_1+a_2e^{-\tau s})\hat e_2^k-\upsilon ^2\partial _{xx}\hat e_2^k&=0, \quad\text{in}\quad\Omega_2\\ 
\partial_x\hat e_2^k(L,s)&=\partial_x\hat e_1^k(L,s),\\
\hat e_2^k(\infty,s)&=0 ,\\ 
\end{array}\right.
\]
This is followed by the update step:\\
$$\hat h^k(s)=\theta\hat e^k_2(L,s)+(1-\theta)\hat h^{k-1}(s).$$
On solving the boundary value problems in the Dirichlet and Neumann step  
we get,
\begin{align*}
\hat e_1^k(x,s) &=\hat h^{k-1}(s)e^{\frac{\sqrt{s+a_1+a_2e^{-\tau s}}(x-L)}{\upsilon}},
\\
\hat e_2^k(x,s) &=-\hat h^{k-1}(s)e^{-\frac{\sqrt{s+a_1+a_2e^{-\tau s}}(x-L)}{\upsilon}}.
\end{align*}
Finally, from the update condition we get,
\begin{align*}
\hat h^k(s) &=(1-2\theta)^k \hat h^0(s).   
\end{align*}
Then the inverse Laplace transform yields,
$$
h^k(t)=(1-2\theta)^k  h^0(t).
$$
So for $\theta$ = $1/2$, the DNWR converges in two iterations. For other values of $\theta$ in (0, 1), it exhibits linear convergence. 
\end{proof}

\subsection{Convergence of NNWR}

We consider the same non-overlapping decomposition of $\Omega$ to apply NNWR to (\ref{eq_1}). Assuming an initial guess $g^0(t)$ along the interface $\Gamma$, we solve Dirichlet subproblem in each $\Omega_i$ in parallel, followed by Neumann solve on each $\Omega_i$. For $k=1,2,\ldots,$ we do:
Dirichlet solve on subdomains $\Omega_i$, $i=1,2$:
\begin{equation}
\left\{\begin{array}{rl}
\partial_t e_i^k-\upsilon ^2\partial _{xx}^2e_i+a_1e_i^k(x, t)+a_2e_i(x, t-\tau )&=0, \ \ \ (x, t)\in \Omega_i\times (0, T), \\ 
 e_i^k(x, t)&=0, \ \ \ (x, t)\in \Omega_i\times[-\tau,0], \\ 
e_i^k&=0, \ \ \  \textnormal{on }  \partial \Omega_i\backslash\Gamma,  \\ 
e_i^k&=g^{k-1},\ \ \  \textnormal{on } \Gamma, 
\end{array}\right.
\end{equation}
Neumann solve on subdomains $\Omega_i$, $i=1,2$:
\begin{equation}
\left\{\begin{array}{rl}
\partial_t \phi_i^k-\upsilon ^2\partial ^2_{xx}\phi_i^k+a_1\phi_i^k(x, t)+a_2\phi_i^k(x, t-\tau )&=0,\ \ \  (x, t)\in \Omega_i\times (0, T),  \\ 
 \phi_i^k(x, t)&=0,\ \ \  (x, t)\in \Omega_i\times[-\tau,0], \\ 
\phi_i^k&=0,\ \ \   \textnormal{on } \partial \Omega_i\backslash\Gamma,  \\ 
\partial _x\phi_i^k&=\partial _xe_1^k-\partial_xe_2^k,\ \ \   \textnormal{on }   \Gamma,  
\end{array}\right.
 \end{equation}

The update condition along the interface is:\\
\begin{equation}
g^{k}(t)=g^{k-1}(t)-\theta( \phi_1^{k}\mid _{\Gamma\times(0, T)} +\phi_2^{k}\mid _{\Gamma\times(0, T)}).
\end{equation}\\
The parameter $\theta\in(0,1)$ is the relaxation parameter.
 We examine whether $g^k(t)$ converges to zero as $k$
 approaches infinity. \\


\begin{theorem}
For Parabolic PDE with time delay (\ref{eq_1}), the NNWR algorithm converges linearly for $0<\theta<1/2$, $\theta\neq 1/4$. The algorithm converges in two iterations for $\theta=1/4$, irrespective of the time window.
\label{thm:thmnnwr}
\end{theorem}
\begin{proof}
 The error equation for $k=1,2,3\ldots$ is transformed into the following form by applying the Laplace transform to the NNWR algorithm,  \\

\[
\left\{\begin{array}{rl}
(s+a_1+a_2e^{-\tau s})\hat e_1^k&=\upsilon ^2\partial _{xx}\hat e_1^k,\quad\text{in}\quad\Omega_1,\\ 
\hat e_1^k(-\infty,s)&=0,\\ 
\hat e_1^k(L,s)&=\hat g^{k-1}(s),\\
\end{array}\right.\ \  
\left\{\begin{array}{rl}
(s+a_1+a_2e^{-\tau s})\hat e_2^k&=\upsilon ^2\partial _{xx}\hat e_2^k,\quad\text{in}\quad\Omega_2,\\ 
\hat e_2^k(\infty,s)&=0,\\ 
\hat e_2^k(L,s)&=\hat g^{k-1}(s),\\
\end{array}\right.
\]\\
\[
\left\{\begin{array}{rl}
(s+a_1+a_2e^{-\tau s})\hat \phi_1^k&=\upsilon ^2\partial _{xx}\hat \phi_1^k,\quad\text{in}\quad\Omega_1,\\ 
\hat \phi_1^k(-\infty,s)&=0,\\ 
\partial_x\hat \phi_1^k(\Gamma)&=\partial_x\hat e_1^k-\partial_x\hat e_2^k,\\
\end{array}\right.
\left\{\begin{array}{rl}
(s+a_1+a_2e^{-\tau s})\hat \phi_2^k&=\upsilon ^2\partial _{xx}\hat \phi_2^k,\quad\text{in}\quad\Omega_2,\\ 
\hat \phi_2^k(\infty,s)&=0,\\ 
-\partial_x\hat \phi_2^k(\Gamma)&=\partial_x\hat e_1^k-\partial_x\hat e_2^k,\\
\end{array}\right.
\]\\

followed by the update step,
\begin{equation}
\hat g^{k}(s)=\hat g^{k-1}(s)-\theta(\hat \phi_1^{k} (L,s)+\hat \phi_2^{k}(L, s)).
\end{equation}
Solving the Dirichlet boundary value problem produces; 

\begin{equation*}
\begin{aligned}
\hat e_1^k(x,s) &=\hat g^{k-1}(s)e^{\frac{\sqrt{s+a_1+a_2e^{-\tau s}}(x-L)}{\upsilon}},
\\
\hat e_2^k(x,s) &=\hat g^{k-1}(s)e^{\frac{-\sqrt{s+a_1+a_2e^{-\tau s}}(x-L)}{\upsilon}}.
\end{aligned}
\end{equation*}
Solving the Neumann boundary problem yields 
\begin{equation*}
\begin{aligned}
\hat \phi_1^k(x,s) &=2\hat g^{k-1}(s)e^{\frac{\sqrt{s+a_1+a_2e^{-\tau s}}(x-L)}{\upsilon}},
\\
\hat \phi_2^k(x,s) &=-2\hat g^{k-1}(s)e^{-\frac{\sqrt{s+a_1+a_2e^{-\tau s}}(x-L)}{\upsilon}}.
\end{aligned}
\end{equation*}

By induction, the relation for the update step along the interface becomes,\\
\begin{equation*}
 \begin{aligned}
\hat g^k(s) &=(1-4\theta)^k \hat g^0(s),
\end{aligned}
\end{equation*}
By inverse Laplace transform, we have\\
$$ g^k(t)= \left ( 1-4\theta \right )^kg^{0}(t).$$\\
So for $\theta=1/4$, the NNWR algorithm converges in 2 iterations. For other $\theta$ in $(0,1/2)$, convergence is linear.
\end{proof}

 \section{Hyperbolic PDEs with Time Delay}
    As the second model problem, we consider a linear wave-type equation with constant time delay (this equation appears in \cite{Rodriguez}),
    $$ u_{tt}=c^2u_{xx}+ \lambda u(x,t-\tau), t>\tau, -a\leq x\leq b $$
with an initial condition
$$
 u(x,t)=\phi (x,t), 0\leq t\leq \tau, -a\leq x\leq b $$
and Dirichlet boundary conditions
\begin{equation}
\label{eq_9}
 u(-a,t)=g(t),u (b,t)=h(t), t\geq 0.
 \end{equation}
where $\lambda$ is a free parameter and $c^2$ is the elasticity coefficient. Similar to the parabolic case, we analyse the error equation so the initial and boundary conditions and history functions are taken as zero.

  \subsection{Convergence Results of DNWR :
  }
    We formulate the DNWR algorithm for the error equations for the problem (\ref{eq_9}). Given an initial guess $h^0(t)$ along the interface \{$x=0$\}, we solve for $k=1,2,\ldots$.

      Dirichlet Part:
\begin{equation}
\left\{\begin{array}{rl}\label{eq_10}
\partial_{tt} e_1^k-c^2\partial_{xx} e_1^k-\lambda e_1^k(x, t-\tau ) &=0, \ \ \ (x, t)\in \Omega_1\times (0, T),  \\ 
  e_1^k(x, t)&=0, \  \ \ (x, t)\in \Omega_1\times [-\tau, 0], \\ 
   \partial_t e_1^k(x, t)&=0, \  \ \ (x, t)\in \Omega_1\times [-\tau, 0], \\ 
e_1^k(-a, t)&=0, \ \ \ t\in(0, T),\\ 
e_1^k(0, t)&=h^{k-1}(t), \ \ \  t\in(0, T).
\end{array}\right.
\end{equation}\\
Neumann Part:
\begin{equation}\label{eq_11}
\left\{\begin{array}{rl}
\partial_{tt} e_2^k-c ^2\partial_{xx} e_2^k-\lambda e_2^k(x, t-\tau )&=0, \ (x, t)\in \Omega_2\times (0, T),  \\ 
  e_2^k(x, t)&=0, \  (x, t)\in \Omega_2\times [-\tau, 0], \\ 
  \partial_t e_2^k(x, t)&=0, \  \ \ (x, t)\in \Omega_1\times [-\tau, 0], \\ 
  \partial_x e_2^k(0, t)&=\partial_x e_1^k(0 , t), \  t\in(0, T),\\
e_2^k(b, t)&=0,  \  t\in(0, T).

\end{array}\right.
\end{equation}\\
The update condition along the interface, for a relaxation parameter $\theta \in (0,1)$ is;\\
\begin{equation}\label{eq_12}
h^{k}( t)=\theta e_2^{k}\mid _{\Gamma\times(0, T)} +(1-\theta) h^{k-1}(t).
\end{equation}\\

  \begin{theorem}

 For equal subdomains, the DNWR method (\ref{eq_10})-(\ref{eq_12}) converges linearly for $0<\theta<1$, $\theta\neq 1/2$. For $\theta=1/2$ it converges within two iterations, irrespective of the time window.
 \label{thm:thmdnwr}
\end{theorem}
\begin{proof}

On applying the Laplace transform to the subproblems (\ref{eq_10})-(\ref{eq_11}) and to the update condition (\ref{eq_12}), we get the subdomain solutions as
\begin{align*}
\hat e_1^k(x,s) &=\frac{\hat h^{k-1}(s)\times \sinh 
 \left(\frac{(x+a)\sqrt{s^2-\lambda ^{-\tau s}}}{c }\right)}{ \sinh\left(\frac{a\sqrt{s^2-\lambda ^{-\tau s}}}{c} \right)},
\\
\\
\hat e_2^k(x,s) &=\frac{\hat h^{k-1}(s)\times \coth\left(\frac{a\sqrt{s^2-\lambda e^{-\tau s}}}{c} \right)\times \sinh \left(\frac{(b-x)\sqrt{s^2-\lambda e^{-\tau s}}}{c}\right)}{ \cosh\left(\frac{b\sqrt{s^2-\lambda e^{-\tau s}}}{c}\right)}.
\end{align*}
\\
By induction, we get the update condition as,
$$
\hat h^k(s)=\left(1-\theta-\theta \coth\left(\frac{a\sqrt{s^2-\lambda e^{-\tau s}}}{c}\right) \times \tanh \left(\frac{b\sqrt{s^2-\lambda e^{-\tau s}}}{c}\right)\right)^k\hat h^0(s).
$$
For symmetrical subdomains, i.e, $a=b$, we have
$$
\hat h^k(s)=(1-2\theta)^k \hat h^0(s),   
$$
Then the inverse Laplace transform yields,
$$
h^k(t)=(1-2\theta)^k  h^0(t).
$$
For $\theta=1/2$, the DNWR converges within two iterations. For other $\theta$ in (0, 1), it converges linearly. 
\end{proof}

\subsection{Convergence results of NNWR }

The NNWR algorithm for two subdomains $\Omega_1=(-a,0)$ and $\Omega_2=(0,b)$, given an initial guess $g^0$ along $\Gamma$, is as follows:\\
 Solve Dirichlet subproblems on subdomain $\Omega_i$, $i=1,2$:
\begin{equation}\label{eq_13}
\left\{\begin{array}{rl}
\partial_{tt} e_i^k-c ^2\partial _{xx}^2e_i-\lambda e_i(x, t-\tau )&=0, \ \ \ (x, t)\in \Omega_i\times (0, T), \\ 
 e_i^k(x, t)&=0, \ \ \ (x, t)\in \Omega_i\times[-\tau,0], \\ 
 \partial_te_i^k(x, t)&=0, \ \ \ (x, t)\in \Omega_i\times[-\tau,0], \\ 
e_i^k&=0, \ \ \  \textnormal{on }  \partial \Omega_i\backslash\Gamma,  \\ 
e_i^k&=g^{k-1},\ \ \  \textnormal{on } \Gamma, 
\end{array}\right.
\end{equation}

Then solve Neumann subproblems on subdomain $\Omega_i$,$i=1,2.$:
\begin{equation}\label{eq_14}
\left\{\begin{array}{rl}
\partial_{tt} \phi_i^k-c ^2\partial ^2_{xx}\phi_i^k-\lambda \phi_i^k(x, t-\tau )&=0,\ \ \  (x, t)\in \Omega_i\times (0, T),  \\ 
 \phi_i^k(x, t)&=0,\ \ \  (x, t)\in \Omega_i\times[-\tau,0], \\ 
 \partial_t\phi_i^k(x, t)&=0,\ \ \  (x, t)\in \Omega_i\times[-\tau,0], \\
 \partial _x\phi_i^k&=\partial _xe_1^k-\partial_xe_2^k,\ \ \   \textnormal{on }   \Gamma,\\
\phi_i^k&=0,\ \ \   \textnormal{on } \partial \Omega_i\backslash\Gamma.
\end{array}\right.
 \end{equation}

 The update condition along the interface is:\\
\begin{equation}\label{eq_15}
g^{k}(t)=g^{k-1}(t)-\theta( \phi_1^{k}\mid _{\Gamma\times(0, T)} +\phi_2^{k}\mid _{\Gamma\times(0, T)}).
\end{equation}\\
\begin{theorem}
For equal subdomains, the NNWR algorithm  (\ref{eq_13})-(\ref{eq_15}) converges linearly for $0<\theta<1/2$, $\theta\neq1/4$ and converges in two iterations for $\theta=1/4$, irrespective of the time window.
\label{thm:thmnnwr}
\end{theorem}

\begin{proof}
Solving the Dirichlet boundary value problem after applying the Laplace transform provides, 
\begin{align*}
\hat e_1^k(x,s) &=\frac{\hat g^{k-1}(s)\times \sinh \left(\frac{(x+a)\sqrt{s^2-\lambda e^{-\tau s}}}{c}\right)}{ \sinh \left(\frac{a\sqrt{s^2-\lambda e^{-\tau s}}}{c} \right)},
\\
\hat e_2^k(x,s) &=\frac{\hat g^{k-1}(s)\times \sinh \left(\frac{(b-x)\sqrt{s^2-\lambda e^{-\tau s}}}{c }\right)}{ \sinh\left(\frac{b\sqrt{s^2-\lambda e^{-\tau s}}}{c} \right)}.
\end{align*}
Solving the Neumann boundary problem yields,
\begin{align*}
\hat \phi_1^k(x,s) &=\frac{\hat g^{k-1}(s)\times \Psi(s)\times \sinh \left(\frac{(x+a)\sqrt{s^2-\lambda e^{-\tau s}}}{c }\right)}{ \sinh\left(\frac{a\sqrt{s^2-\lambda e^{-\tau s}}}{c}\right)},
\\
\hat \phi_2^k(x,s) &=\frac{\hat g^{k-1}(s)\times \Psi(s)\times \sinh \left(\frac{(b-x)\sqrt{s^2-\lambda e^{-\tau s}}}{c}\right)}{ \sinh\left(\frac{b\sqrt{s^2-\lambda e^{-\tau s}}}{c}\right)},
\end{align*}
where,

$$
 \Psi(s)=\left( \coth\left(\frac{a\sqrt{s^2-\lambda e^{-\tau s}}}{c}\right) + \coth\left(\frac{b\sqrt{s^2-\lambda e^{-\tau s}}}{c }\right)\right).
$$

By induction, the update step along the interface becomes,\\
$$
\hat g^k(s)=\hat g^{k-1}(s)\left [ 1-\theta\left(2+\frac{\tanh\left(a\frac{\sqrt{s^2-\lambda e^{-\tau s}}}{c }\right)}{\tanh\left(b\frac{\sqrt{s^2-\lambda e^{-\tau s}}}{c}\right)}+\frac{\tanh\left(b\frac{\sqrt{s^2-\lambda e^{-\tau s}}}{\upsilon }\right)}{\tanh\left(a\frac{\sqrt{s^2-\lambda e^{-\tau s}}}{c }\right)}\right) \right ].
$$
For symmetrical subdomains, i.e, $a=b$,\\
$$\hat g^k(s)=\left ( 1-4\theta \right )^k\hat g^{0}(s),$$\\
By inverse Laplace transform we have\\
$$ g^k(t)= \left ( 1-4\theta \right )^k g^{0}(t).$$\\
So, for $\theta=1/4$, we get convergence in 2 iterations. For other $\theta$ in $(0,1/2)$, convergence is linear.
\end{proof}

\section{Neutral PDE}
For the third model problem, we consider a Neutral PDE (this equation appears in \cite{vales}),
$$\frac{\partial u(x,t)}{\partial t}=\mu^2\frac{\partial ^2u(x,t)}{\partial x^2}+\mu^2c^2\frac{\partial ^2u(x,t-\tau)}{\partial x^2}+ru(x,t)+du(x,t-\tau)$$

with a history function and Dirichlet boundary conditions,\\ 
\begin{equation}  
\label{eq:neut}
\left\{\begin{array}{rl}
u(x,t) &=H(x,t),\textnormal{ for }  (x,t) \in  [-a,b]\times [-\tau,0],\\ 
u(0,t) &=u(\pi,t)=0, t\geq 0.
\end{array}\right.
\end{equation}

  \subsection{DNWR for Neutral PDE}
  We study the DNWR algorithm for the error equation associated with equation (\ref{eq:neut}) for two subdomains $\Omega_1=(-a,0)$ and $\Omega_2=(0,b)$.\\
Dirichlet Part for subdomain $\Omega_1$ :

\begin{equation}
    \label{eq_20}
   \left \{ \begin{array}{rl}
        \partial_t e_1^k-\mu^2\partial_{xx} e_1^k-\mu^2c^2\partial_{xx}e_1^k(x,t-\tau)-re_1^k(x, t)-de_1^k(x, t-\tau ) &=0,\quad\text{in}\quad \Omega_1\times (0, T),  \\ 
        e_1^k(x, t)&=0,\quad\text{in}\quad\Omega_1\times [-\tau, 0], \\ 
        e_1^k(-a, t)=0,\  \ e_1^k(0, t)&=h^{k-1}(t), \; \; t\in(0, T). \end{array} 
        \right.
\end{equation} 
Neumann Part for subdomain $\Omega_2$ :
\begin{equation}\label{eq_21}
\left\{\begin{array}{rl}
\partial_t e_2^k-\mu^2\partial_{xx} e_2^k-\mu
^2c^2\partial_{xx}e_2^k(x,t-\tau)-re_2^k(x, t)-de_2^k(x, t-\tau )&=0,\quad\text{in}\quad \Omega_2\times (0, T),  \\ 
  e_2^k(x, t)&=0, \quad\text{in}\quad \Omega_2\times [-\tau, 0], \\ 
  \partial_x e_2^k(0 , t)=\partial_x e_1^k(0 , t),\ \ e_2^k(b, t)&=0, \ \ \ \  t\in(0, T),
\end{array}
\right.
\end{equation}\\
The update condition along the interface is:\\
\begin{equation}\label{eq_22}
h^{k}(t)=\theta e_2^{k}\mid _{\Gamma\times(0, T)} +(1-\theta) h^{k-1}(t).
\end{equation}\\

\begin{theorem}
For equal subdomains, the DNWR algorithm for (\ref{eq:neut}) converges linearly for $0<\theta<1$, $\theta\neq 1/2$ and converges in two iterations for $\theta=1/2$, irrespective of the time window.
\label{thm:thmnnwr}
\end{theorem}
\begin{proof}
On applying the Laplace transform to the subproblems (\ref{eq_20})-(\ref{eq_21}) and to the update condition (\ref{eq_22}), the error equation is transformed into the following form,
\[
\left\{\begin{array}{rl}
(s-r-de^{-\tau s})\hat e_1^k-\mu^2(1+c^2e^{-\tau s})\partial _{xx}\hat e_1^k&=0,\quad\text{in}\quad\Omega_1
\\ 
\hat e_1^k(-a,s)&=0 ,\\ 
\hat e_1^k(0,s)&=\hat h^{k-1}(s),\\
\end{array}\right.
\]
\\
\[
\left\{\begin{array}{rl}
(s-r-de^{-\tau s})\hat e_2^k-\mu^2(1+c^2e^{-\tau s})\partial _{xx}\hat e_2^k&=0, \quad\text{in}\quad\Omega_2\\ 
\hat e_2^k(b,s)&=0 ,\\ 
\partial_x\hat e_2^k(0,s)&=\partial_x\hat e_1^k(0,s),\\
\end{array}\right.
\]\\
followed by the update step:\\
$$\hat h^k(s)=\theta\hat e^k_2(0,s)+(1-\theta)\hat h^{k-1}(s).$$\\
On solving boundary value problems in the Dirichlet and Neumann step  
we get,
\begin{align*}
 \hat e_1^k &=\frac{\hat h^{k-1}(s)\sinh\left(\frac{1}{\mu} \sqrt{\frac{s-r-de^{-\tau s}}{1+c^2e^{-\tau s}}}(x+a)\right)}{\sinh\left(\frac{1}{\mu} \sqrt{\frac{s-r-de^{-\tau s}}{1+c^2e^{-\tau s}}}a\right)},\\
\\
\hat e_2^k &=-\frac{\hat h^{k-1}(s)\coth\left(\frac{1}{\mu} \sqrt{\frac{s-r-de^{-\tau s}}{1+c^2e^{-\tau s}}}a\right)\sinh\left(\frac{1}{\mu} \sqrt{\frac{s-r-de^{-\tau s}}{1+c^2e^{-\tau s}}}(b-x)\right)}{\cosh\left(\frac{1}{\mu} \sqrt{\frac{s-r-de^{-\tau s}}{1+c^2e^{-\tau s}}}b\right)}.
\end{align*}
\\
By induction, we get the update condition as,
\begin{align*}
\hat h^k(s) &=\left(1-\theta-\theta \frac{\coth\left(\frac{1}{\mu} \sqrt{\frac{s-r-de^{-\tau s}}{1+c^2e^{-\tau s}}}a\right)}{\coth\left(\frac{1}{\mu} \sqrt{\frac{s-r-de^{-\tau s}}{1+c^2e^{-\tau s}}}b\right)}\right)^k\hat h^0(s).
\end{align*}
For symmetrical domain, $a=b$, 
$$
\hat h^k(s)=(1-2\theta)^k \hat h^0(s).  
$$
Then the inverse Laplace transform yields,
$$
h^k(t)=(1-2\theta)^k  h^0(t).
$$
So, for $\theta=1/2$, the method converges in two iterations. For other $\theta$ in (0, 1), convergence is linear.
\end{proof}
\subsection{NNWR for Neutral PDE}
The NNWR algorithm for 2 subdomains $\Omega_1=(-a,0)$ and
 $\Omega_2=(0,b)$ is as follows:\\
Dirichlet subproblem on subdomains $\Omega_i$,
 \\
\begin{equation}\label{eq_23}
\left\{\begin{array}{rl}
\partial_t e_i^k-\mu^2\partial_{xx} e_i^k-\mu^2c^2\partial_{xx}e_i^k(x,t-\tau)-re_i^k(x, t)&=de_i^k(x, t-\tau ), \quad\text{in}\quad\Omega_i\times (0, T),  \\ 
  e_i^k(x, t)&=0, \quad\text{in}\quad \Omega_i\times [-\tau, 0], \\ 
e_i^k(x, t)&=0, \ \ \ t\in(0, T), x\in \partial \Omega_i \setminus \Gamma,\\ 
e_i^k(0, t)&=g^{k-1}(t), \ \ \  t\in(0, T).
\end{array}\right.
\end{equation}\\
Neumann subproblem on subdomains $\Omega_i$,
\begin{equation}\label{eq_24}
\left\{\begin{array}{rl}
\partial_t \phi_i^k-\mu^2\partial_{xx} \phi_i^k-\mu^2c^2\partial_{xx}\phi_i^k(x,t-\tau)-r\phi_i^k(x, t)&=d\phi_i^k(x, t-\tau ), \quad\text{in}\quad \Omega_i\times (0, T), \\ 
  \phi_i^k(x, t)&=0, \quad\text{in}\quad \Omega_i\times [-\tau, 0], \\ 
\phi_i^k(x, t)&=0,  \  t\in(0, T), x\in (\partial \Omega_i \setminus \Gamma), \\
\partial_x \phi_i^k(\Gamma , t)&=\partial_x e_1^k-\partial_x e_1^k, \  t\in (0, T).
\end{array} \right.
\end{equation}\\
The update condition along the interface is:\\
\begin{equation}\label{eq_25}
g^{k}(t)=g^{k-1}(t)-\theta( \phi_1^{k}\mid _{\Gamma\times(0, T)} +\phi_2^{k}\mid _{\Gamma\times(0, T)}).
\end{equation}\\
\begin{theorem}
For equal subdomains, the NNWR algorithm for (\ref{eq:neut}) converges linearly for $0<\theta<1/2$, $\theta\neq1/4$ and converges in two iterations for $\theta=1/4$, irrespective of the time window.
\label{thm:thmnnwr}
\end{theorem}
\begin{proof}
After applying Laplace transform to the equations (\ref{eq_23})-(\ref{eq_25}), the transformed equations become:
    \[
\left\{\begin{array}{rl}
(s-r-de^{-\tau s})\hat e_1^k-\mu^2(1+c^2e^{-\tau s})\partial _{xx}\hat e_1^k&=0,\quad\text{in}\quad\Omega_1
\\ 
\hat e_1^k(-a,s)&=0 ,\\ 
\hat e_1^k(0,s)&=\hat g^{k-1}(s),\\
\end{array}\right.
\]
\[
\left\{\begin{array}{rl}
(s-r-de^{-\tau s})\hat e_2^k-\mu^2(1+c^2e^{-\tau s})\partial _{xx}\hat e_2^k&=0, \quad\text{in}\quad\Omega_2\\ 
\hat e_2^k(b,s)&=0 ,\\ 
\hat e_1^k(0,s)&=\hat g^{k-1}(s),\\
\end{array}\right.
\]
\\
\[
\left\{\begin{array}{rl}
(s-r-de^{-\tau s})\hat \phi_1^k-\mu^2(1+c^2e^{-\tau s})\partial _{xx}\hat \phi_1^k&=0,\quad\text{in}\quad\Omega_1
\\ 
\hat \phi_1^k(-a,s)&=0 ,\\ 
\partial _x\hat \phi_1^k&=\partial _x\hat e_1^k-\partial_x\hat e_2^k,\\
\end{array}\right.
\]
\\
\[
\left\{\begin{array}{rl}
(s-r-de^{-\tau s})\hat \phi_2^k-\mu^2(1+c^2e^{-\tau s})\partial _{xx}\hat \phi_2^k&=0, \quad\text{in}\quad\Omega_2\\ 
\hat \phi_2^k(b,s)&=0 ,\\ 
-\partial _x\hat \phi_2^k&=\partial _x \hat e_1^k-\partial_x \hat 
e_2^k,\\
\end{array}\right.
\]\\
Solving the Dirichlet boundary value problems gives, \\
\begin{align*}
\hat e_1^k &=\frac{\hat g^{k-1}(s)\sinh\left(\frac{1}{\mu} \sqrt{\frac{s-r-de^{-\tau s}}{1+c^2e^{-\tau s}}}(x+a)\right)}{\sinh\left(\frac{1}{\mu} \sqrt{\frac{s-r-de^{-\tau s}}{1+c^2e^{-\tau s}}}a\right)},
\\
\\
\hat e_2^k &=\frac{\hat g^{k-1}(s)\sinh\left(\frac{1}{\mu} \sqrt{\frac{s-r-de^{-\tau s}}{1+c^2e^{-\tau s}}}(b-x)\right)}{\sinh\left(\frac{1}{\mu} \sqrt{\frac{s-r-de^{-\tau s}}{1+c^2e^{-\tau s}}}b\right)}.
\end{align*}
\\
The solutions to the Neumann boundary problems are: 

\begin{align*}
\hat \phi_1^k(x,s) &=\frac{\hat g^{k-1}(s)\xi(s)\sinh\left(\frac{1}{\mu} \sqrt{\frac{s-r-de^{-\tau s}}{1+c^2e^{-\tau s}}}\right)(x+a))}{\cosh\left(\frac{1}{\mu} \sqrt{\frac{s-r-de^{-\tau s}}{1+c^2e^{-\tau s}}}a\right)},
\\
\\
\hat \phi_2^k(x,s) &=\frac{\hat g^{k-1}(s)\xi(s)\sinh\left(\frac{1}{\mu} \sqrt{\frac{s-r-de^{-\tau s}}{1+c^2e^{-\tau s}}}(b-x)\right)}{\cosh\left(\frac{1}{\mu} \sqrt{\frac{s-r-de^{-\tau s}}{1+c^2e^{-\tau s}}}b\right)},
\end{align*}

where,
\\
$$
\xi (s)=\left (\coth\left ( \frac{1}{\mu} \sqrt{\frac{s-r-de^{-\tau s}}{1+c^2e^{-\tau s}}} a\right )+\coth\left(\frac{1}{\mu} \sqrt{\frac{s-r-de^{-\tau s}}{1+c^2e^{-\tau s}}}b\right)\right ).
$$\\
\\
By induction, the relation for the update step along the interface where $\theta$
being a relaxation parameter is:\\
$$
\hat g^k=\hat g^{k-1}\left [ 1-\theta \left ( 2+\frac{\coth\left(\frac{1}{\mu} \sqrt{\frac{s-r-de^{-\tau s}}{1+c^2e^{-\tau s}}}a\right)}{\coth\left(\frac{1}{\mu} \sqrt{\frac{s-r-de^{-\tau s}}{1+c^2e^{-\tau s}}}b\right)} +\frac{\coth\left(\frac{1}{\mu} \sqrt{\frac{s-r-de^{-\tau s}}{1+c^2e^{-\tau s}}}b\right)}{\coth\left(\frac{1}{\mu} \sqrt{\frac{s-r-de^{-\tau s}}{1+c^2e^{-\tau s}}}a\right)}\right ) \right ].
$$
For symmetrical domains, $a=b$,\\
$$\hat g^k(s)=\left ( 1-4\theta \right )^k\hat g^{0}(s),$$\\
Using inverse Laplace transform, we have\\
$$ g^k(t)= \left ( 1-4\theta \right )^k g^{0}(t)
.$$\\
So, for $\theta=1/4$, we get convergence in 2 iterations. For other $\theta$ in $(0,1/2)$, convergence is linear.
\end{proof}

\section{Numerical Illustrations}\label{sec2}
In this section, we conduct numerical experiments to assess the effectiveness of DNWR and NNWR for all model problems discussed above (\ref{eq_1}), (\ref{eq_9}), (\ref{eq:neut}).
We work on an error equation so the boundary conditions and history function become zero. 
For a symmetrical case, we divide the domain into two equal subdomains, and the initial transmission condition chosen on the interface boundary is $h(t)=t^2$.
For Reaction-Diffusion equation with time delay, we compare DNWR and NNWR with Classical Schwarz and Optimized Schwarz methods \cite{shulinquasi}. Also, to empirically verify the effectiveness of DNWR and NNWR in the presence of multiple subdomains, we perform the experiment with multiple subdomains for the Reaction-Diffusion equation with time delay. 

\subsection{Reaction-Diffusion Equation with Time Delay}
For the numerical implementation, we take the spatial domain as $\Omega=(0,6)$ and the time window $T=6$. We use a backward Euler in time with $\Delta t=0.1$ and the central finite difference in space with $\Delta x= 0.1 $.

\begin{figure}[!h]
    \centering
    \subfloat{\includegraphics[width=0.462\linewidth]{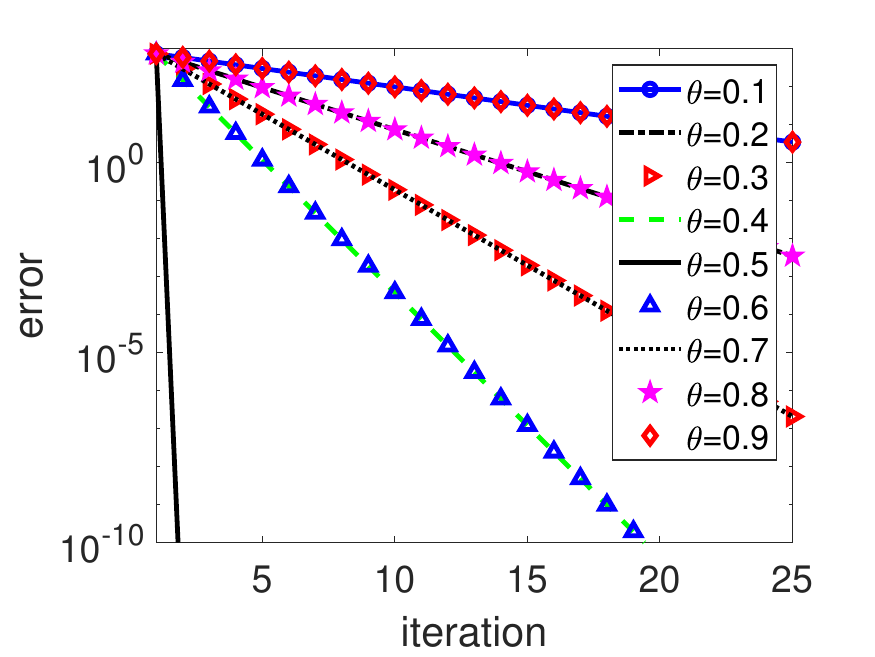} }%
    \qquad
    \subfloat{\includegraphics[width=0.462\linewidth]{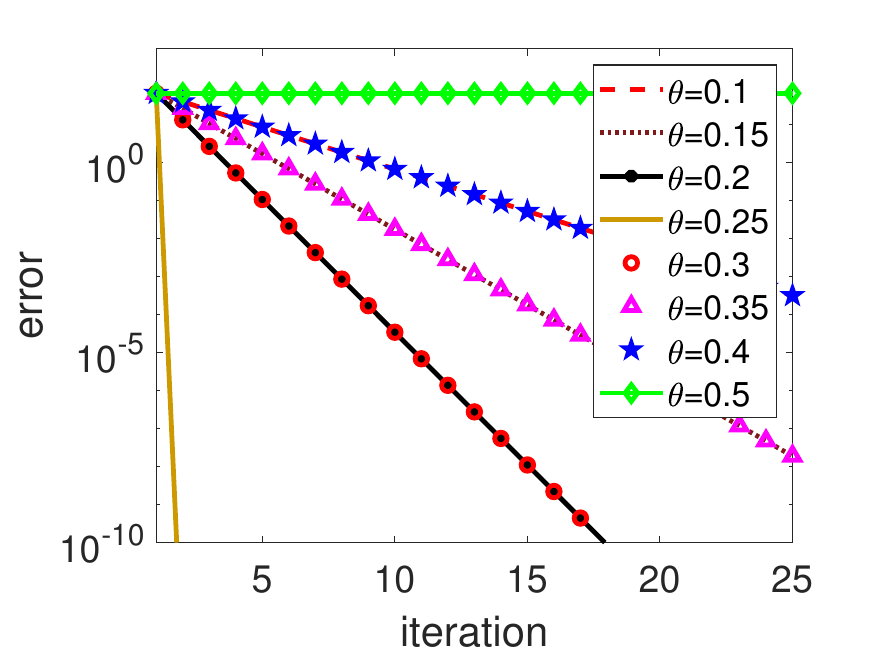} }%
    \caption{Parabolic PDE with time delay when $a_1 \neq 0$: Convergence for different relaxation parameters; Left: DNWR algorithm, Right: NNWR algorithm. }
    \label{dnwrnnwreq1}
    \end{figure}

\subsubsection*{Case 1: $a_1\neq0$ in Equation (\ref{eq_1})}

For this experiment, the value of coefficients in Equation (\ref{eq_1}) are taken as $a_1=1$, $a_2=2.3$, $\nu=1$ and $\tau=1.5$ as used in \cite{shulin}. We first run the experiments for symmetrical domains for both DNWR and NNWR; the result is reported in Fig.\ref{dnwrnnwreq1}. We get convergence in two iterations for $\theta=1/2$ in DNWR, and for other $\theta $ lying between 0 to 1, convergence is linear, which is similar to the result obtained in Theorem 1. 
 In NNWR for $\theta=1/4$, we get convergence in 2 iterations, and for other $\theta $ lying between 0 to 1/2, we see linear convergence, which is what we obtained in Theorem 2. For asymmetrical domains, we consider the subdomains $\Omega_1=(0,4)$ and subdomain $\Omega_2=(4,6)$ and plot the convergence results in Fig \ref{dnwrnnwrneq1}.

 \begin{figure}[!h]
    \centering
    \subfloat{{\includegraphics[width=0.462\linewidth]{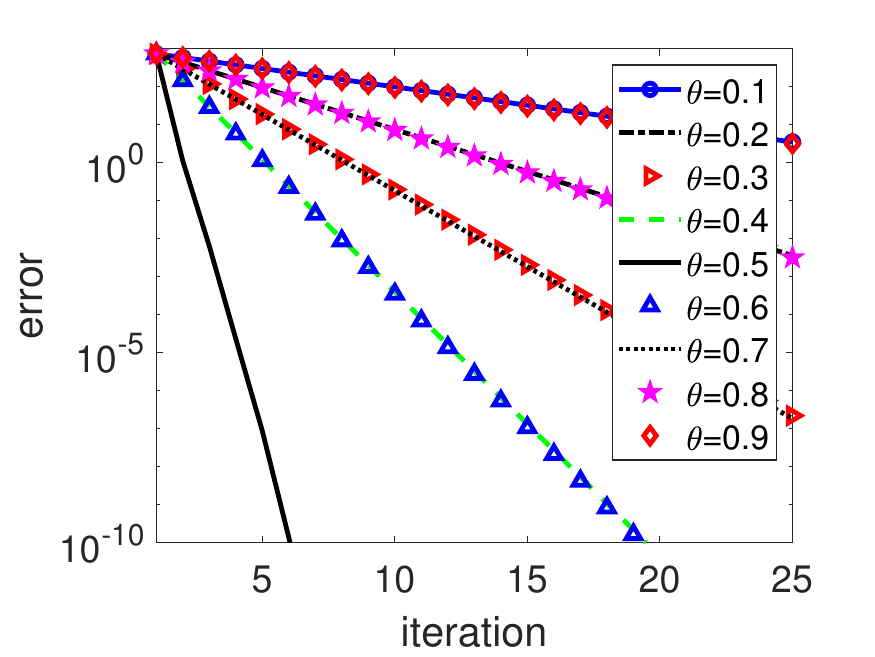} }}%
    \qquad
    \subfloat{{\includegraphics[width=0.462\linewidth]{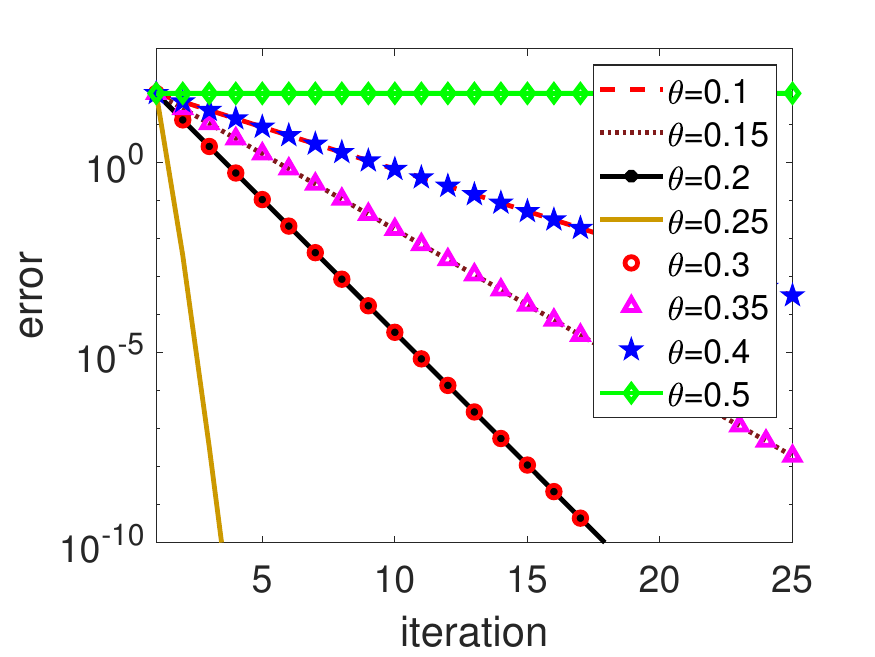} }}%
    \caption{Parabolic PDE with time delay when $a_1 \neq 0$: Left; Convergence of DNWR for different parameters for larger Dirichlet domain; Right: Convergence of NNWR for different parameters. }
    \label{dnwrnnwrneq1}
    \end{figure}

      We then compare DNWR and NNWR for asymmetrical case with Classical Schwarz and Optimized Schwarz methods having overlapping as $2\Delta x$  with two different optimal parameters $p=p^*_{old}$ and $p=p^*_{new}$ where $p^*_{old}=4.21342663346202$ and $p^*_{new}=1.7141023669459$  (as derived by Shu-lin et al. in \cite{shulin}) and the result obtained is presented in
Fig. \ref{comppare1_1}. It shows that both DNWR and NNWR are better in terms of iteration efficiency as compared to Classical Schwarz and Optimised Schwarz methods for optimal choice of the parameter $\theta$.

 \begin{figure}[!h]
    \centering
    \includegraphics[width=0.492 \linewidth]{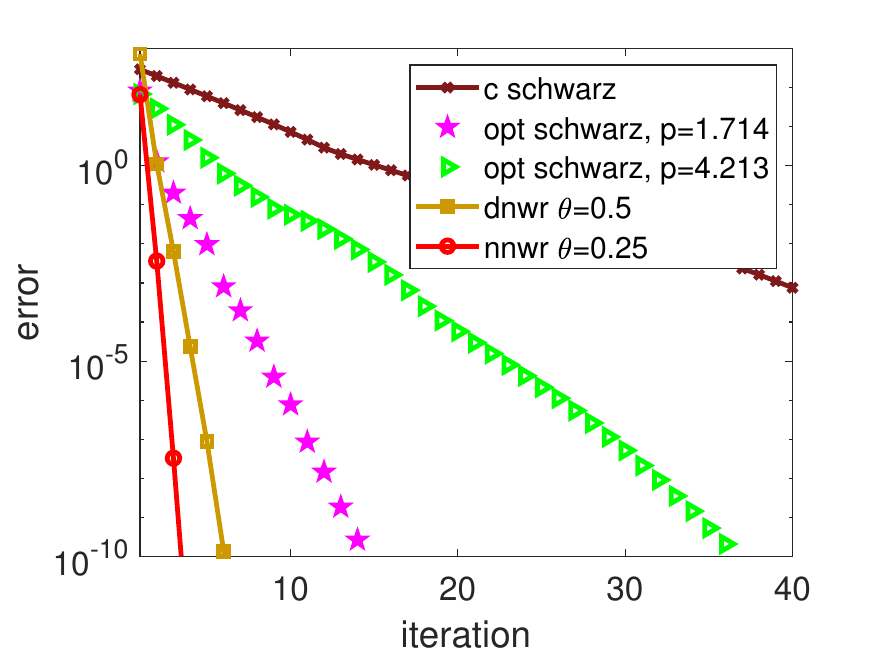}
    \caption{ Parabolic PDE with time delay when $a_1\neq0$; Comparison of NNWR and DNWR  with classical Schwarz and Optimised Schwarz method.} 
     \label{comppare1_1}
    \end{figure}

\subsubsection*{Case 2: $a_1=0$ in Equation (\ref{eq_1})}
In this case, for the numerical experiments, we take values $a_2=0.028$, $\nu =1$ and $\tau=3$ in problem \eqref{eq_1} as used in \cite{shulin}. We first examine the equal subdomain case for both DNWR and NNWR, and the outcome is presented in Fig. \ref{dnwrnnwr_eq}. For symmetrical domains, we get convergence in just two iterations for $\theta=1/2$ in DNWR 
and for $\theta=1/4$ in NNWR.
 \begin{figure}[!h]
    \centering
    \subfloat{{\includegraphics[width=0.462\linewidth]{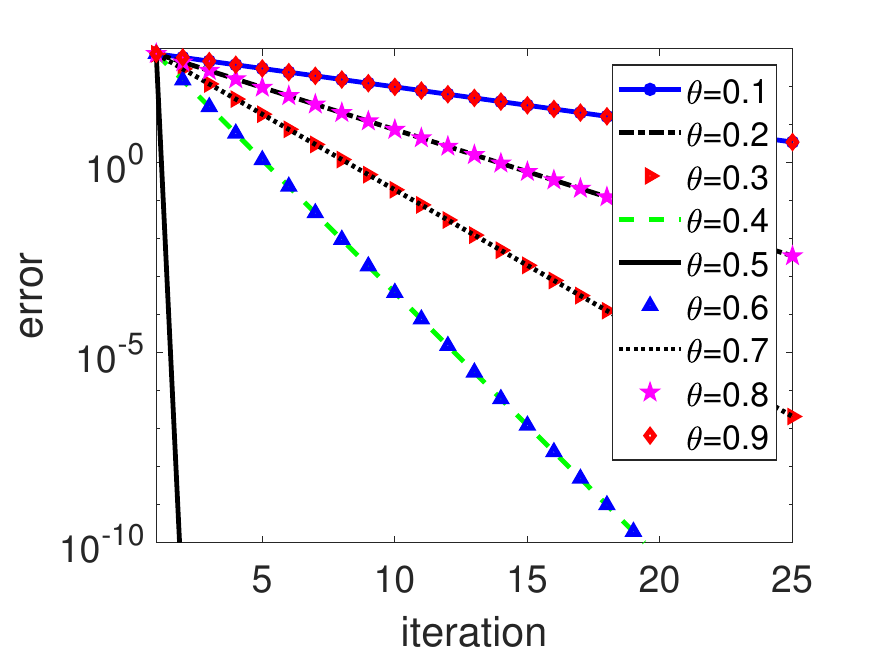} }}%
    \qquad
    \subfloat{{\includegraphics[width=0.462\linewidth]{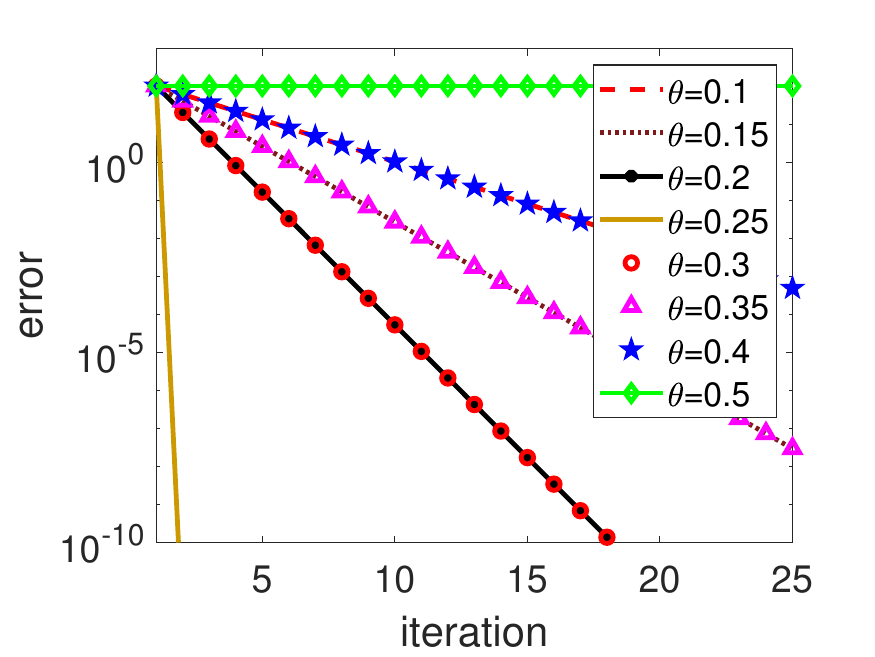} }}%
    \caption{Parabolic PDE with time delay when $a_1=0$; Left: Convergence of DNWR for different $\theta$; Right: Convergence of NNWR for different $\theta$. }
    \label{dnwrnnwr_eq}
    \end{figure}
    
 We plot the convergence curves for asymmetrical domains in Fig. \ref{dnwrnnwr_neq}, and we find superlinear convergence for $\theta=1/2$ for DNWR and $\theta=1/4$ for NNWR. 
\begin{figure}[!h]
    \centering
    \subfloat{{\includegraphics[width=0.462\linewidth]{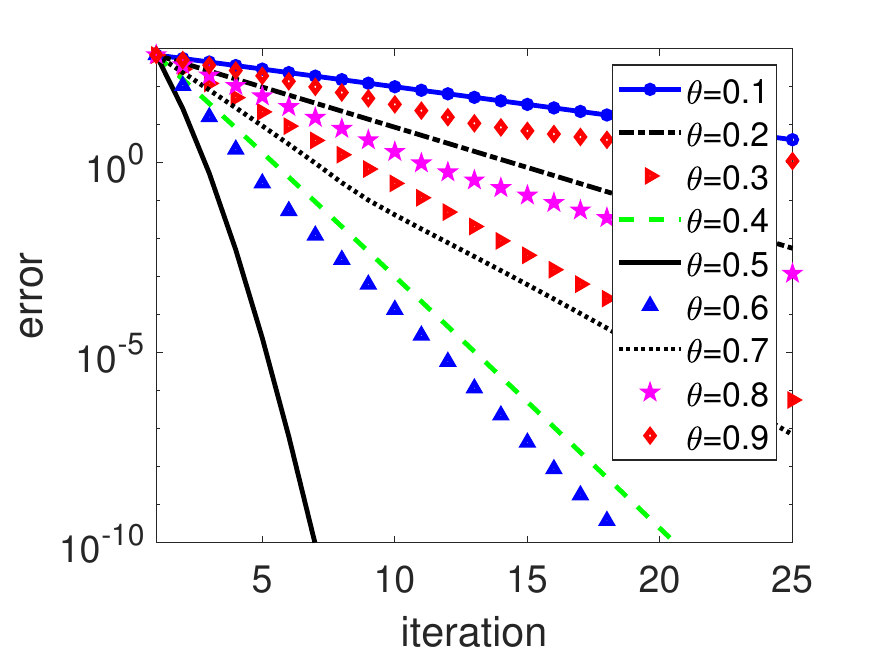} }}%
    \qquad
    \subfloat{{\includegraphics[width=0.462\linewidth]{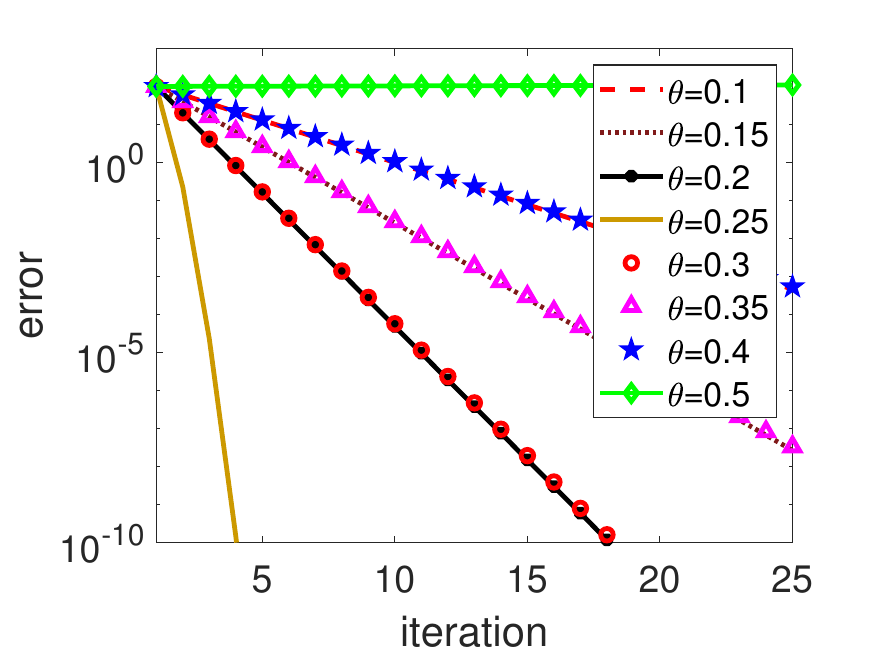} }}%
    \caption{Parabolic PDE with time delay when $a_1=0$; Left: Convergence of DNWR for larger  Dirichlet domain, Right: Convergence of NNWR. }
    \label{dnwrnnwr_neq}
    \end{figure}
    Comparing DNWR and NNWR with the Classical Schwarz and Optimised Schwarz approaches, Fig. \ref{comp} illustrates which approach takes fewer iterations for convergence.

     \begin{figure}[!h]
    \centering
    \includegraphics[width=0.45 \linewidth]{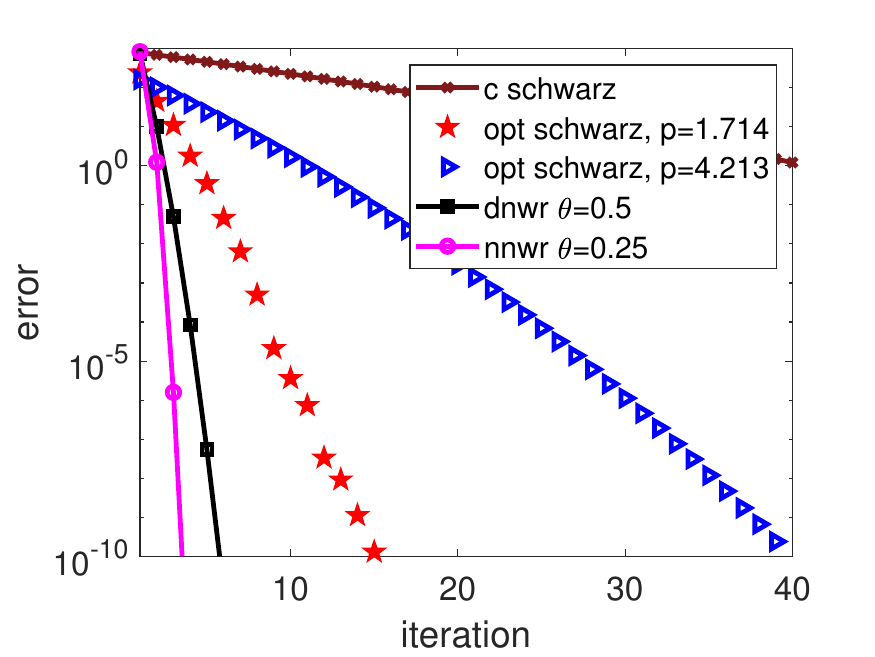}
    \caption{Parabolic PDE with time delay when $a_1=0$: Comparison of NNWR and DNWR  with classical Schwarz and Optimised Schwarz method.} 
     \label{comp}
    \end{figure}

    \subsection{Wave equation with Time Delay}
  For numerical implementation, we use an implicit finite difference scheme, the central difference in space and time with $\Delta x=0.1$ and $\Delta t=0.1$. We take the spatial domain $\Omega=(0,6)$ and time window $T=6$, delay $\tau$ is $3$. See Fig. \ref{dnwrnnwr_eq_wave} for the convergence of DNWR and NNWR for the wave equation when subdomain sizes are equal.
 \begin{figure}[!h]
    \centering
    \subfloat{{\includegraphics[width=0.462\linewidth]{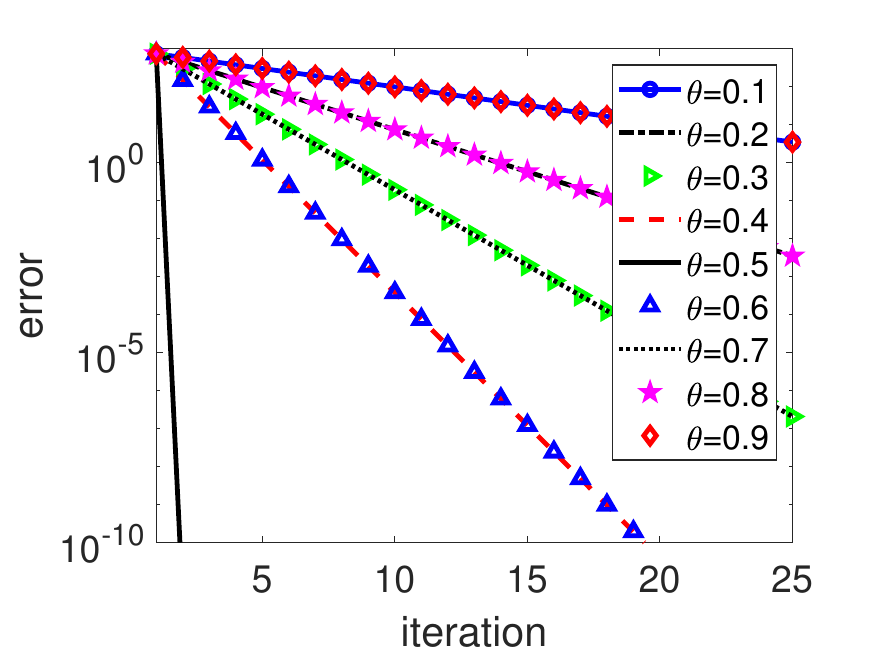} }}%
    \qquad
    \subfloat{{\includegraphics[width=0.462\linewidth]{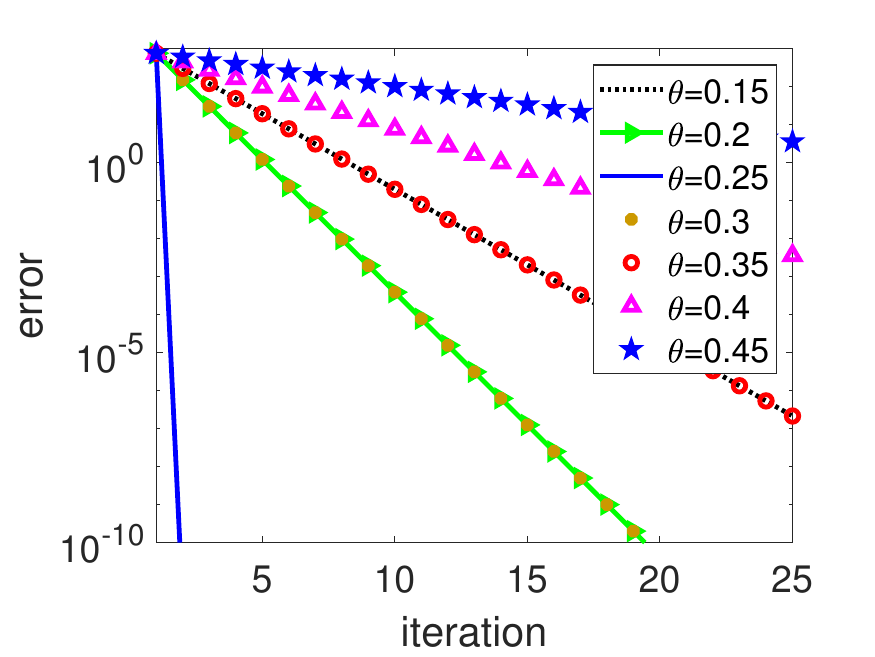} }}%
    \caption{Wave PDE with time delay: Left: Convergence of DNWR for different $\theta$, Right: Convergence of NNWR for different $\theta$. }
    \label{dnwrnnwr_eq_wave}
    \end{figure}

 Fig. \ref{dnwrnnwr_neq_wave} illustrates how DNWR method performs for unequal subdomains. For $\theta=1/2$, we obtain superlinear convergence. We also obtain superlinear convergence for $\theta=1/4$ for NNWR method with asymmetrical domains.
 
 \begin{figure}[!h]
    \centering
    \subfloat{{\includegraphics[width=0.462\linewidth]{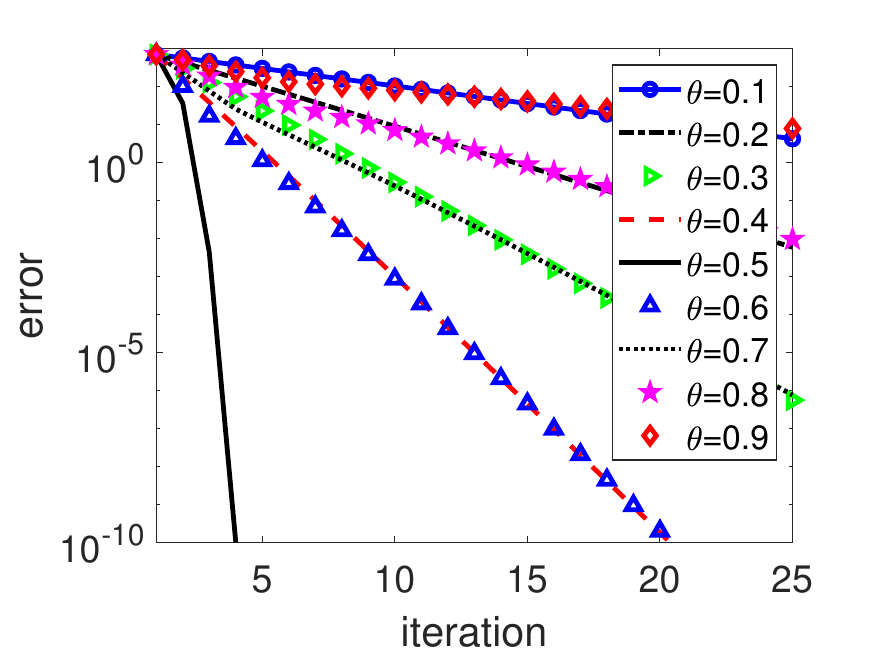} }}
    \qquad
    \subfloat{{\includegraphics[width=0.462\linewidth]{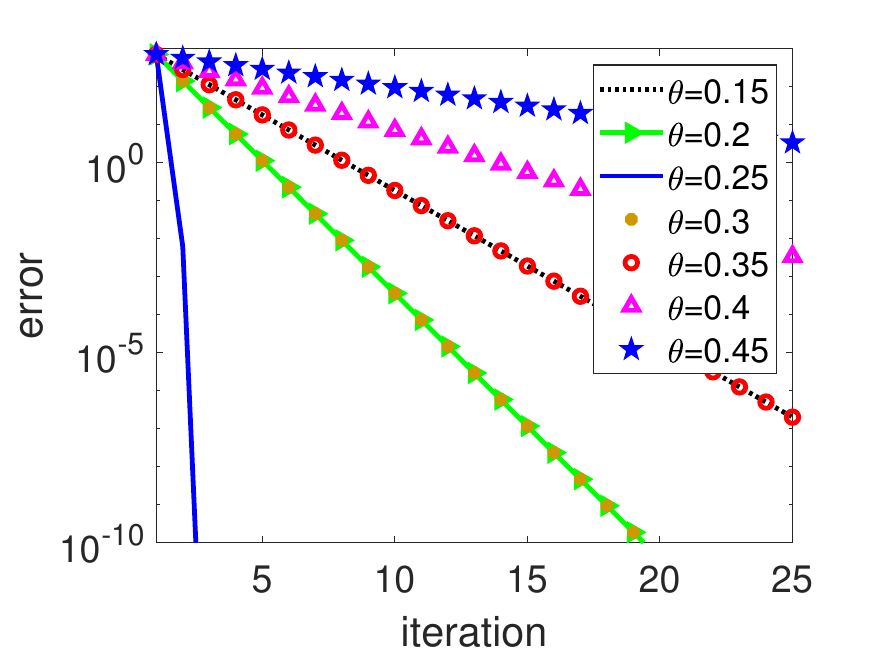} }}
    \caption{Wave PDE with time delay: Left: Convergence of DNWR for larger Dirichlet domain, Right: Convergence of NNWR for larger first domain. }
    \label{dnwrnnwr_neq_wave}
    \end{figure}
    
\subsection{Neutral PDE}

To check the effectiveness of DNWR and NNWR for Equation  (\ref{eq:neut}), we run experiments for the corresponding error equation with homogeneous boundary conditions and zero history function. As per \cite{vales}, we take $\mu =1$, $c=0.1$, $r=0.05$ and $d=\frac{cr}{2}$. We determine the solution for the time interval [0,5] using the value of $\tau = 1$, $\Delta x=0.1$ and $\Delta t=0.1$. 
Convergence of DNWR algorithm for symmetrical domains is obtained in two iterations for $\theta=1/2$  and for other $\theta\in(0,1)$  we get linear convergence. For NNWR $\theta=1/4$ act as the optimal parameter 
(see Fig. \ref{dnwrnnwr_eq_neutral}).  
We also test the methods for asymmetrical division of spatial domain. For a Dirichlet domain $(\Omega_1 = (0, 4.5))$  larger than the Neumann domain $\Omega_2 = (4.5, 6)$, with $\theta=1/2$, the DNWR algorithm exhibits superlinear convergence; see Fig.\ref{dnwrnnwr_neq_neutral}. The same is true for NNWR with $\theta=1/4$. 
\begin{figure}[!h]
    \centering
    \subfloat{{\includegraphics[width=0.462\linewidth]{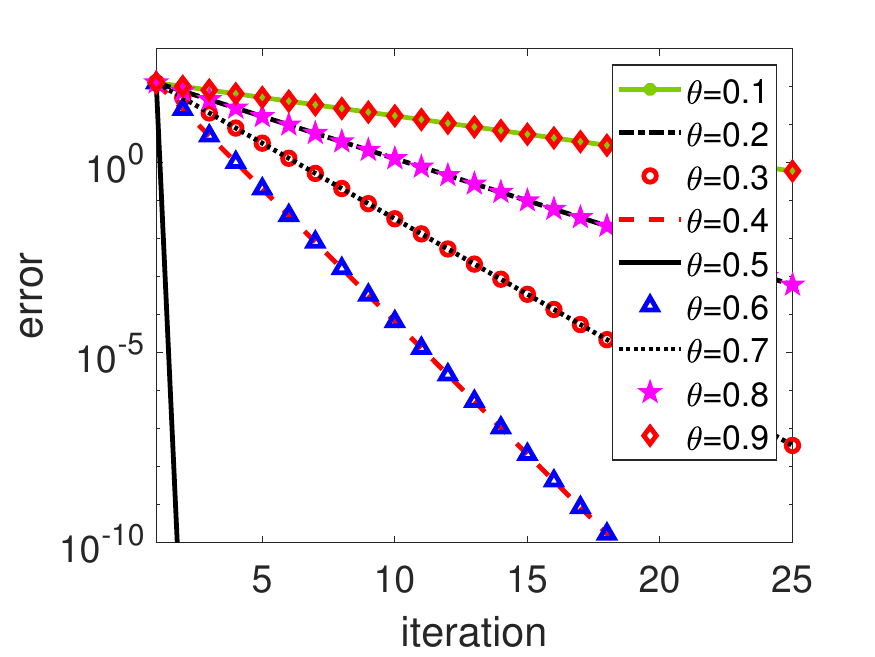} }}%
    \qquad
    \subfloat{{\includegraphics[width=0.462\linewidth]{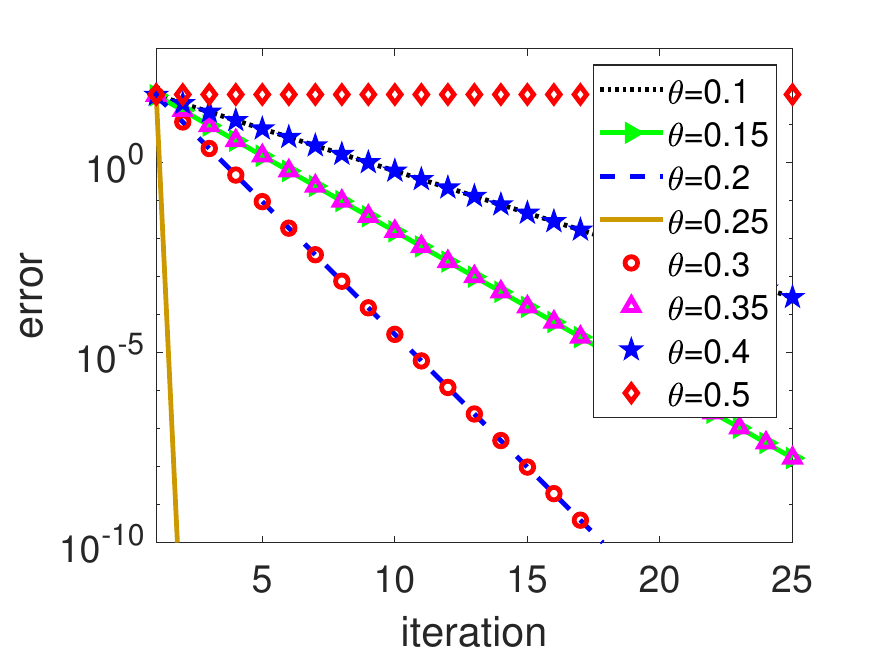} }}%
    \caption{Neutral PDE: Left: Convergence of DNWR for different $\theta$, Right: Convergence of NNWR for different $\theta$. }
    \label{dnwrnnwr_eq_neutral}
    \end{figure}
  
 \begin{figure}[!h]
    \centering
    \subfloat{{\includegraphics[width=0.462\linewidth]{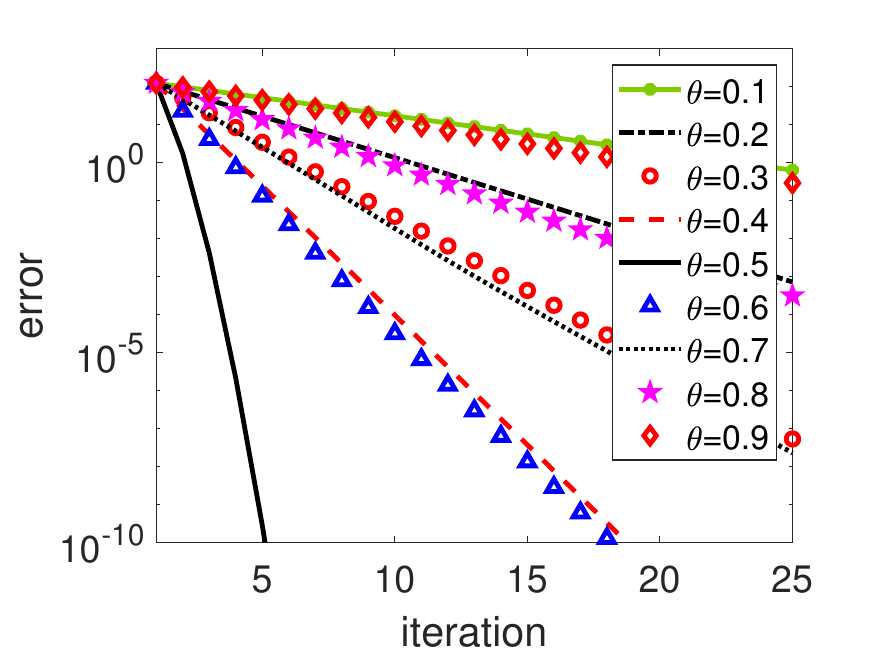} }}
    \qquad
    \subfloat{{\includegraphics[width=0.462\linewidth]{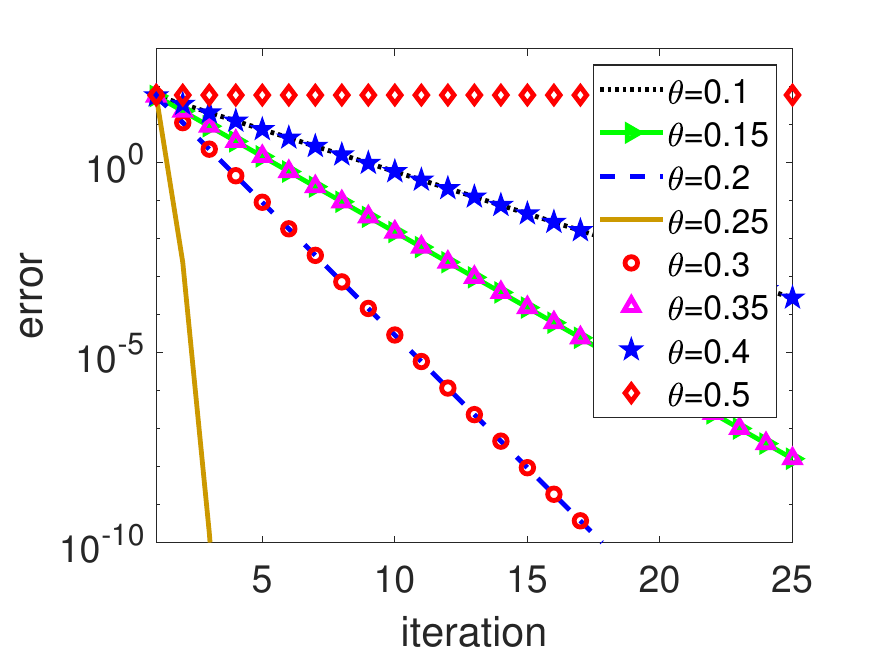} }}
    \caption{Neutral PDE: Left: Convergence of DNWR for larger  Dirichlet domain, Right: Convergence of NNWR for larger first domain. }
    \label{dnwrnnwr_neq_neutral}
    \end{figure}

To summarise, in DNWR, the algorithm performs best when $\theta= 1/2$, while in NNWR, the best choice is $\theta=1/4$. This holds true for both symmetrical and asymmetrical cases, indicating superior performance compared to other values of $\theta$ for which convergence is linear.

    \subsection{Experiments for Multiple Subdomain case}

       For this experiment, we consider the Reaction-Diffusion equation with delay (\ref{eq_1}), with $a_1=0$ and delay $\tau=0.03$.
        We divide the domain $[0,5]\times[0,0.1]$ into multiple subdomains; the capacity to parallelize increases with the number of subdomains.  For discretization, we take $\Delta t=0.002$, and each spatial subdomain contains 10 grid points, so $\Delta x$ varies with the number of subdomains. 
        In DNWR for multi-subdomain, one arrangement is to solve the Dirichlet problem in the middle subdomain first, followed by the Dirichlet-Neumann solution in the adjacent subdomain (see Figure \ref{fig:arr_3}).  For convergence of DNWR for different numbers of subdomains when $\theta =1/2$, see Fig. \ref{dnwrnnwr_multishulin}.
         \begin{figure}[!h]
     \centering
     \includegraphics[width=0.6 \linewidth]{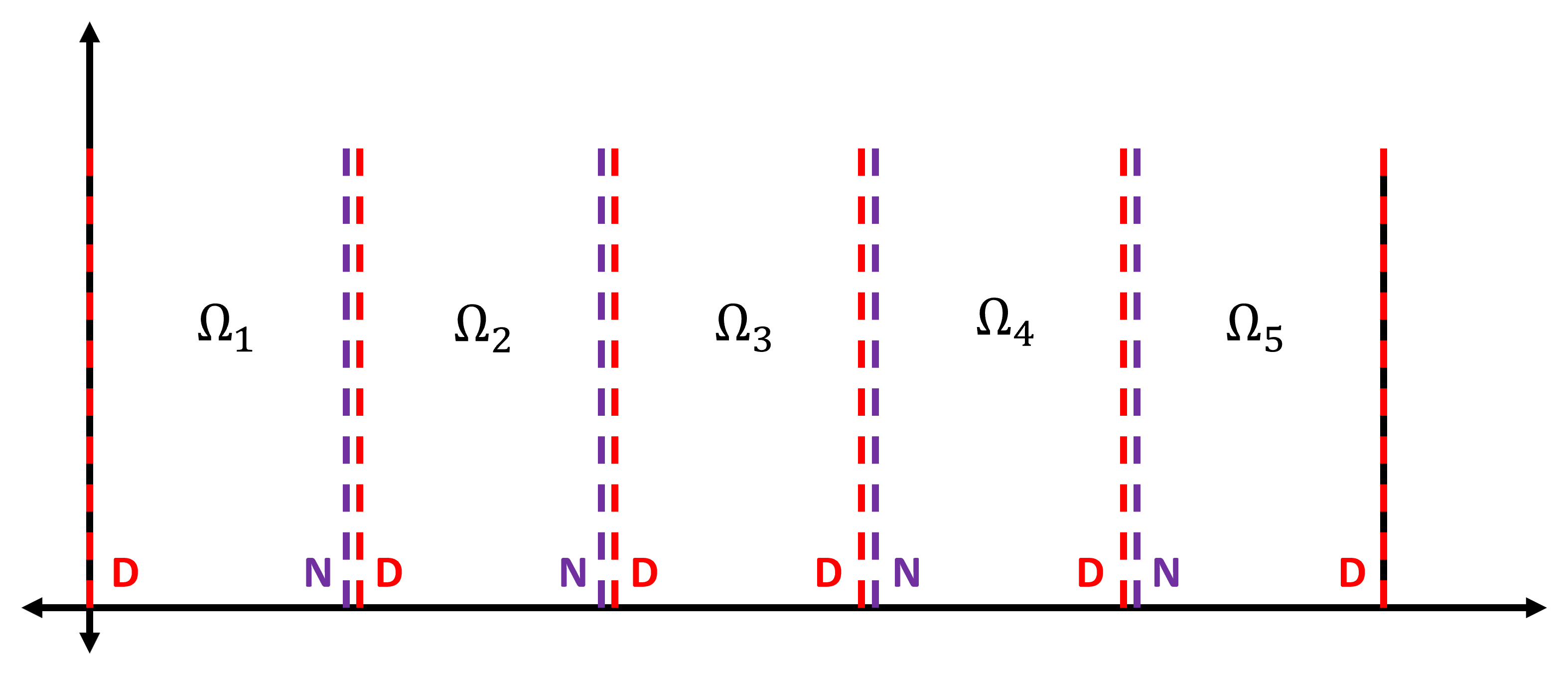}
     \caption{Arrangement of boundary conditions in DNWR. }
     \label{fig:arr_3}
 \end{figure}
 \\
 NNWR for multiple subdomains is a natural extension of the two subdomain case. First, we solve the Dirichlet problem in each subdomain by assuming an initial guess on the interface boundaries, which is then followed by a correction step involving Neumann solving on each subdomain. For convergence of NNWR for different number of subdomains when $\theta =1/4$, see Fig. \ref{dnwrnnwr_multishulin}.\\

\begin{figure}[!h]
    \centering
    \subfloat{{\includegraphics[width=0.462\linewidth]{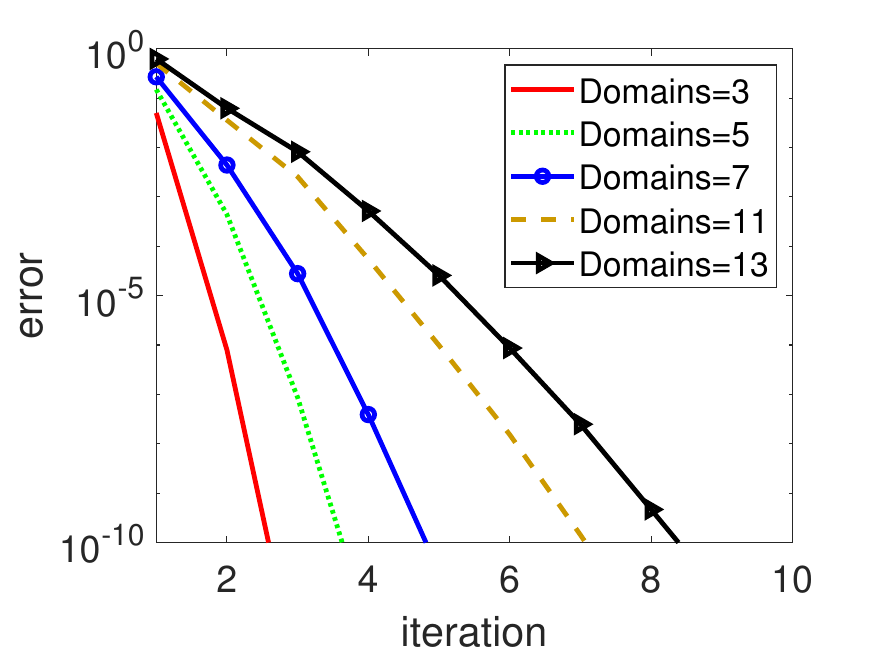} }}
    \qquad
    \subfloat{{\includegraphics[width=0.462\linewidth]{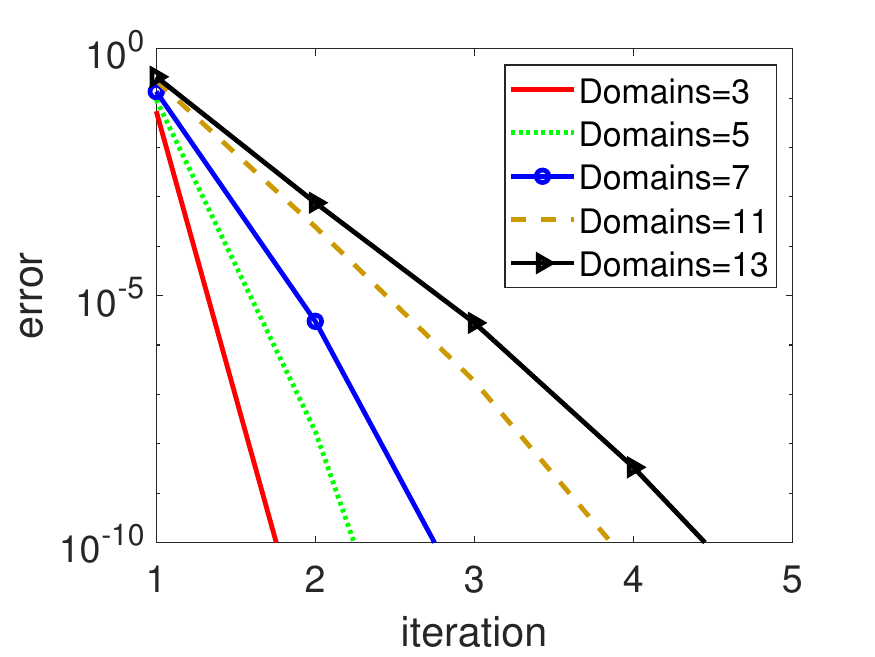} }}
    \caption{Parabolic PDE with time delay when $a_1=0$: Comparision of convergence behaviour for different number of subdomains, Left: DNWR, Right: NNWR. }
    \label{dnwrnnwr_multishulin}
    \end{figure}

 \section{Conclusion}\label{sec3}
 We extend and examine, parallelisable methods DNWR and NNWR for PDE with time delay, especially for the Reaction-Diffusion equation with time delay, the Hyperbolic delay PDE and the Neutral PDE. We established a convergence estimate in one dimension for DNWR and NNWR and also performed numerical experiments. Our findings indicate that a particular choice of the relaxation parameter ensures superlinear convergence both for DNWR and NNWR, and we get convergence in just two iterations when $\theta=1/2$ for DNWR, while NNWR accomplishes the same in two iterations for $\theta =1/4$. 
 We conducted convergence analysis for a symmetrical case of hyperbolic delay partial differential equation (PDE) and Neutral PDE in one dimension; we theoretically proved convergence in two iterations for both DNWR and NNWR algorithms when  $\theta=1/2$ and $\theta=1/4$, respectively. This validation was further confirmed through numerical experiments.
 These new substructuring methods demonstrate significantly better convergence rates for the model problem (\ref{eq_1}) than optimized and quasi-optimized Schwarz methods. Finally, we observed that the number of iterations needed for convergence increases in both DNWR and NNWR methods as the number of subdomains increases.\\~\\
  \textbf{Acknowledgements} The authors would like to thank IIT Bhubaneswar for the research facility.
\\~\\
 \textbf{Data Availability Statement } All the data that was produced or generated during the course of the research has been included in the manuscript.
 \section*{Declarations}
 
 \textbf{Conflict of interest} The authors declare no Conflict of interest.


\bibliographystyle{unsrt}

\end{document}